\documentclass[12pt]{article}
\usepackage[utf8]{inputenc}
\usepackage[T1]{fontenc}
\usepackage{amssymb}
\usepackage{amsthm}
\usepackage{amsmath}
\usepackage{graphicx}
\usepackage{color}
\usepackage{amsfonts}
\usepackage{epsfig}
\usepackage{latexsym}
\usepackage{graphicx}
\usepackage{relsize}
\newtheorem{theorem}{Theorem}[section]
\newtheorem{lemma}[theorem]{Lemma}

\usepackage[parfill]{parskip}
\begin{document}

\title{\textbf{ The Unit-Gompertz Distribution:  Characterizations  and Properties }}

\author{
        M. Z. Anis\\
\footnotesize{\it{SQC \& OR Unit,  Indian Statistical Institute,}}\\
\footnotesize{\it{203, B. T. Road, Calcutta 700 108, India}}\\
\footnotesize{\it{E-mail:} {\rm zafar@isical.ac.in }}
}
\date{}
\maketitle

\begin{abstract}
In a recent paper, Mazucheli et al. (2019) introduced the unit-Gompertz (UG) distribution and studied some of its properties. In a complementary work, Anis and De (2020) corrected some of the  subtle errors in the original paper and studied some other interesting properties of this new distribution. However, to the best of our knowledge no charcterization results on this distribution have appeared in the literature. This is addressed in the present paper using truncated moments; and some more properties are investigated. 
\end{abstract}

\section{Introduction}

Data are generated in all branches of social, biological, physical and engineering sciences. They are modelled by means of probability distributions for better understanding. It is important and necessary that an appropriate probability distribution be fitted to the empirical data so that meaningful and correct conclusions can be drawn.

In this connection characterization results have been used to test goodness of fit for probability distributions. Marchetti and Mudholkar (2002) show that characterization theorems can be natural, logical and effective starting points for constructing goodness-of-fit tests.  Nikitin (2017) observes that tests based on characterization results are usually more efficient than other tests. Goodness-of-fit tests based on new characterizations  results abound in the literature. Baringhaus and Henze (2000) study two new omnibus goodness of fit tests for exponentiality, each based on a characterization of the exponential distribution via the mean residual life function. Akbari (2020) presents characterization results and new
goodness-of-fit tests based on these new characterizations for Pareto distribution. Earlier,  Gl\"{a}nzel (1987) derived characterization theorems for some families of both continuous and discrete distributions and used  them   as a basis for parameter estimation. 

The main purpose of the present work is to present characterization results for the Unit-Gompertz distribution introduced by Mazucheli et al. (2019). Essentially, this new distribution is derived from the  Gompertz distribution. Recall that the 
density function of the  Gompertz distribution is given by 
\[g\left( y\mid \alpha, \beta\right) = \alpha \beta\exp \left( \alpha+\beta y - \alpha e^{\beta y}\right),  \]
where $ y>0;  $ and $ \alpha > 0 $ and $ \beta >0 $ are shape and scale parameters, respectively. Using the transformation 
 \[ X= e^{-Y}, \]
this new distribution with support on $ \left( 0, 1\right),  $ which is referred to as the \emph{unit-Gompertz distribution} is obtained. Its  pdf and cdf are  given by 
\begin{equation}
f\left( x\mid\alpha, \beta\right) = \frac{\alpha\beta\exp \left[ -\alpha\left( 1/x^{\beta}-1\right) \right] }{x^{1+\beta}}; \mbox {~~~}\alpha>0, \beta>0, x \in (0, 1) \label{pdf1}
\end{equation}
and 
\begin{equation}
F\left( x\mid\alpha, \beta\right) =\exp \left[ -\alpha\left( 1/x^{\beta}-1\right) \right] \label{cdf} 
\end{equation}	
respectively.

 Mazucheli et al. (2019) used this new distribution to model the maximum flood level (in millions of cubic feet per second) for Susquehanna River at
 Harrisburg, Pennsylvania (reported in Dumonceaux and Antle (1973)) and tensile strength of polyester fibers as given in  Quesenberry
 and Hales (1980).  As an application, Jha et al. (2020) consider the problem of estimating multicomponent stress-strength reliability under progressive Type II censoring when stress and strength variables follow unit Gompertz distributions with common scale parameter. Jha et al. (2019) consider reliability estimation in a multicomponent stress–strength based on unit-Gompertz distribution. Inferential issues have also been studied. In this connection, mention may be made of  Kumar et al. (2019) who are concerned with inference for the unit-Gompertz model based on record values and inter-record times.  Anis and De (2020) not only correct some of the subtle errors in the original paper of Mazucheli et al. (2019) but also discuss reliability properties and stochastic ordering among others. However, no characterization results of this new distribution have been studied.

This paper attempts to fill in this gap and is organized as follows. Characterization of the distribution by truncated moments is presented in Section 2. Some more properties not investigated earlier are presented in Section 3.   Section 4 concludes the paper. 

\section{Characterizations}
We shall now give two characterizations of the Unit-Gompertz distribution based on truncated first moment. To prove the charcterization results, we shall need two lemmas and an assumption which are presented first.
\subsection{The Preleminaries}
	\textbf{Assumption $\mathcal{A}$:} Assume that $ X $ is an absolutely continuous random variable with the pdf given in (\ref{pdf1}) and the corresponding cdf given in (\ref{cdf}) above. Further, we also assume that $ E\left( X\right)  $ exits and the density $ f\left( x\right)  $ is differentaible. Define $ \beta = \sup \left\lbrace x: F\left( x\right) <1\right\rbrace  $ and $ \alpha = \inf \left\lbrace x: F\left( x\right) >0\right\rbrace. $

	\begin{lemma} \label{L1}
		Under the Assumption $\mathcal{A}$, if $E\left( X\mid X \leq x\right) = g\left( x\right) \tau  \left( x\right),  $ where  $ g\left( x\right) $ is a continuous differentiable function of $ x$ with the condition that $\int_{\alpha }^{x}\frac{u-g^{\prime }(u)}{g(u)}du$ is finite for all $ x> \alpha, $ and $ \tau  \left( x\right)=\frac{f\left( x\right) }{F\left( x\right)} $  then
		\[f\left( x\right) =c\exp\left[ {\int \frac{x-g^{\prime }\left( x\right) }{g\left( x\right) }dx}\right] , \] where   the constant  $ c$ is determined by the condition $\int_{\alpha}^{\beta }f\left( x\right)dx=1.$\\
	\end{lemma}
	
	\begin{lemma}\label{L2}
		Under the Assumption $\mathcal{A}$, if $E\left( X\mid X \geq  x\right) = h\left( x\right) r \left( x\right),  $  where  $ h\left( x\right) $  is a continuous differentiable function of $x$ with
		the condition that $\int_{\alpha }^{x}\frac{u-h^{\prime }(u)}{h(u)}du$ is
		finite for all $ x>\alpha, $ and $ r\left( x\right)=\frac{f\left( x\right) }{1- F\left( x\right)} $   then
		\[f\left( x\right)=c\exp\left[- {\int \frac{x+h^{\prime }\left( x\right) }{h\left( x\right) }dx}\right] ,\]
		where $c$ is a
		constant determined by the condition $\int_{\alpha }^{\beta }f\left( x\right)dx=1.$
	\end{lemma}
	
	See Ahsanullah (2017) for details of proof of these two lemmas.
	\subsection{ Characterization Theorems}
	We shall now state and prove two characterization theorems based on the truncated first moment.
	\begin{theorem}
		Suppose that the random variable $X$ satisfies the assumption $\mathcal{A}$ with $\alpha =0$ and $\beta =1.$ Then $E\left( X\mid X \leq x\right) = g\left( x\right) \tau  \left( x\right),  $ where $\tau (x)=\frac{f(x)}{F(x)}$ and		
		\begin{equation}
	g\left( x\right) =\frac{\alpha^{1/\beta}}{\alpha\beta}\exp \left( \frac{\alpha}{x^{\beta}}\right) x^{1+\beta}\Gamma\left( \left\lbrace 1-\frac{1}{\beta}\right\rbrace ; \frac{\alpha}{x^{\beta}}\right) \label{gx}
		\end{equation}
		
		where $ \Gamma\left( s;x\right)  $ is the upper incomplete gamma function defined by 
		\begin{equation}
		\Gamma\left( s; x\right) =\int_{x}^{\infty}t^{s-1}e^{-t} dt,\label{inc-up-gamma}
		\end{equation}
		if and only if 		
		\begin{equation}
		f\left( x\mid\alpha, \beta\right) = \frac{\alpha\beta\exp \left[ -\alpha\left( 1/x^{\beta}-1\right) \right] }{x^{1+\beta}}; \mbox {~~~}\alpha>0, \beta>0, x \in (0, 1). \label{pdf}
		\end{equation}
	
	\end{theorem}
	
	\begin{proof}
		Suppose $$ f\left( x\mid\alpha, \beta\right) = \frac{\alpha\beta\exp \left[ -\alpha\left( 1/x^{\beta}-1\right) \right] }{x^{1+\beta}}; \mbox {~~~}\alpha>0, \beta>0, x \in (0, 1). $$
		We have 
		\[g\left( x\right) \tau  \left( x\right)= E\left( X\mid X \leq x\right) = \frac{1}{F\left( x\right)}\int_{0}^{x}t f\left( t\right)dt\]		
		Since $\tau (x)=\frac{f(x)}{F(x)},$ it  follows that 
		\begin{eqnarray*}
		g\left( x\right) f\left( x\right) &=&\int_{0}^{x}t f\left( t\right)dt\\
		&=& \int_{0}^{x}t \frac{\alpha\beta\exp \left[ -\alpha\left( 1/t^{\beta}-1\right) \right] }{t^{1+\beta}}dt\\
		&=&e^{\alpha}\alpha^{1/\beta}\Gamma\left( \left\lbrace  1-\frac{1}{\beta}\right\rbrace  ; \frac{\alpha}{x^{\beta}}\right),  
		\end{eqnarray*}
	where $ \Gamma\left( s;x\right)  $ is the upper incomplete gamma function defined in (\ref{inc-up-gamma}) above. Hence, after simplifying, we obtain
	\[ g\left( x\right) =\frac{\alpha^{1/\beta}}{\alpha\beta}\exp \left( \frac{\alpha}{x^{\beta}}\right) x^{1+\beta}\Gamma\left( \left\lbrace 1-\frac{1}{\beta}\right\rbrace ; \frac{\alpha}{x^{\beta}}\right) . \] 
	Conversely, suppose that $  g\left( x\right) $ is given by (\ref{gx}). Differentiating  $g\left( x\right)$ with respect to $ x, $ and simplifying, we obtain
	\begin{eqnarray*}
	 g^{\prime}\left( x\right)&=& x-  g\left( x\right)\left[ \frac{\alpha \beta}{x^{\beta +1}} - \frac{\beta +1}{x}\right]. \\
	\end{eqnarray*}
Hence, 
\[ \frac{x-g^{\prime}\left( x\right)}{g\left( x\right)}= \frac{\alpha \beta}{x^{\beta +1}} - \frac{\beta +1}{x}.\]
By Lemma \ref{L1}, we have
\begin{equation}
\frac{f^{\prime}\left( x\right)}{f\left( x\right)}= \frac{\alpha \beta}{x^{\beta +1}} - \frac{\beta +1}{x}.\label{f-eqn1}
\end{equation}
Integrating both sides of (\ref{f-eqn1}) with respect to $ x, $ we obtain
\begin{equation}
f\left( x\right)=k\frac{\exp \left( -\alpha x^{-\beta }\right) }{x^{\beta +1}},
\end{equation}
where $ k $ is a constant. Using the condition $ \int_{0}^{1} f\left( x\right)dx=1,$ we get 
\begin{equation*}
f\left( x\right) = \frac{\alpha\beta\exp \left[ -\alpha\left( 1/x^{\beta}-1\right) \right] }{x^{1+\beta}}
\end{equation*}
This completes the proof.
	\end{proof}

	\begin{theorem}
		Suppose that the random variable $X$ satisfies the assumption $\mathcal{A}$ with $\alpha =0$ and $\beta =1.$ Then $E\left( X\mid X  \geq x\right) = h\left( x\right) r \left( x\right),  $ where $r(x)=\frac{f(x)}{1-F(x)}$ and		
		\begin{equation}
		h\left( x\right) =\frac{\alpha^{1/\beta}}{\alpha\beta}\exp \left( \frac{\alpha}{x^{\beta}}\right) x^{1+\beta}\left[\Gamma\left( \left\lbrace 1-\frac{1}{\beta}\right\rbrace ;\alpha \right)- \Gamma\left( \left\lbrace 1-\frac{1}{\beta}\right\rbrace ; \frac{\alpha}{x^{\beta}}\right)\right]  \label{hx}
		\end{equation}		
		where $ \Gamma\left( s;x\right)  $ is the upper incomplete gamma function defined in (\ref{inc-up-gamma}) above, if and only if 
				\begin{equation}
		f\left( x\mid\alpha, \beta\right) = \frac{\alpha\beta\exp \left[ -\alpha\left( 1/x^{\beta}-1\right) \right] }{x^{1+\beta}}; \mbox {~~~}\alpha>0, \beta>0, x \in (0, 1). \label{pdf2}
		\end{equation}
		
	\end{theorem}
	
	\begin{proof}
		Suppose $$ f\left( x\mid\alpha, \beta\right) = \frac{\alpha\beta\exp \left[ -\alpha\left( 1/x^{\beta}-1\right) \right] }{x^{1+\beta}}; \mbox {~~~}\alpha>0, \beta>0, x \in (0, 1). $$
		We have 
		\[h\left( x\right) r\left( x\right)= E\left( X\mid X \geq x\right) = \frac{1}{1-F\left( x\right)}\int_{x}^{1}t f\left( t\right)dt\]		
		Since $r(x)=\frac{f(x)}{1-F(x)},$ it  follows that 
		\begin{eqnarray*}
			h\left( x\right) f\left( x\right) &=&\int_{x}^{1}t f\left( t\right)dt\\
			&=& E\left( X\right) - \int_{0}^{x}t f\left( t\right)dt\\
			&=& E\left( X\right) -\int_{0}^{x}t \frac{\alpha\beta\exp \left[ -\alpha\left( 1/t^{\beta}-1\right) \right] }{t^{1+\beta}}dt\\
			&=&e^{\alpha}\alpha^{1/\beta}\left[ \Gamma\left( \left\lbrace  1-\frac{1}{\beta}\right\rbrace  ;\alpha  \right)-\Gamma\left( \left\lbrace  1-\frac{1}{\beta}\right\rbrace  ; \frac{\alpha}{x^{\beta}}\right)\right] ,  
		\end{eqnarray*}
		where $ \Gamma\left( s;x\right)  $ is the upper incomplete gamma function defined in (\ref{inc-up-gamma}) above. Hence,		
		\[ 	h\left( x\right) =\frac{\alpha^{1/\beta}}{\alpha\beta}\exp \left( \frac{\alpha}{x^{\beta}}\right) x^{1+\beta}\left[\Gamma\left( \left\lbrace 1-\frac{1}{\beta}\right\rbrace ;\alpha \right)- \Gamma\left( \left\lbrace 1-\frac{1}{\beta}\right\rbrace ; \frac{\alpha}{x^{\beta}}\right)\right].  \]
		Conversely, suppose that $  h\left( x\right) $ is given by (\ref{hx}). Differentiating  $h\left( x\right)$ with respect to $ x, $ and simplifying, we obtain
		\begin{eqnarray*}
			h^{\prime}\left( x\right)&=& -x-  h\left( x\right)\left[ \frac{\alpha \beta}{x^{\beta +1}} - \frac{\beta +1}{x}\right]. \\
		\end{eqnarray*}
		Hence, 
		\[ -\frac{x+h^{\prime}\left( x\right)}{h\left( x\right)}= \frac{\alpha \beta}{x^{\beta +1}} - \frac{\beta +1}{x}.\]
		By Lemma \ref{L2}, we have
		\begin{equation}
		\frac{f^{\prime}\left( x\right)}{f\left( x\right)}= \frac{\alpha \beta}{x^{\beta +1}} - \frac{\beta +1}{x}.\label{f-eqn}
		\end{equation}
		Integrating both sides of (\ref{f-eqn}) with respect to $ x, $ we obtain
		\begin{equation}
		f\left( x\right)=k\frac{e^{-\alpha x^{-\beta }}}{x^{\beta +1}},
		\end{equation}
		where $ k $ is a constant. Using the condition $ \int_{0}^{1} f\left( x\right)dx=1,$ we get 
		\begin{equation*}
		f\left( x\right) = \frac{\alpha\beta\exp \left[ -\alpha\left( 1/x^{\beta}-1\right) \right] }{x^{1+\beta}}
		\end{equation*}
		This completes the proof.
	\end{proof}

\section{Properties}

Most of the important properties of this distribution have been considered in Mazucheli et al. (2019) and the complementary paper by Anis and De (2020). For completeness we consider the $  L-\mathrm{moments} $ and two new measures of entropy.
\subsection{$ L-\mathrm{moments} $}
 $ L-\mathrm{moments} $ are summary statistics for probability distributions and data samples and are computed from linear combinations of the ordered data values (hence the prefix $L$). Hosking (1990) has shown that the $ L-\mathrm{moments} $ posses  theoretical advantages over ordinary moments. Moreover, they are less sensitive to outliers compared to the conventional moments. Computation of the first few sample $ L-\mathrm{moments} $ and $ L-\mathrm{moment} $ ratios of a data set provides a useful summary of the location, dispersion, and shape of the distribution from which the sample was drawn. They  can be used to obtain reasonably efficient estimates of parameters when a distribution is fitted to the data.  The main advantage of  $ L-\mathrm{moments} $  over conventional moments is that  $ L-\mathrm{moments}, $  being linear functions of the data, suffer less from the effects of sampling variability;  are more robust than conventional moments to outliers in the data and enable more secure inferences to be made from small samples about an underlying probability distribution.  $ L-\mathrm{moments} $  sometimes yield more efficient parameter estimates than the maximum likelihood estimates.\\
 These $ L-\mathrm{moments} $ can be defined in terms of \emph{probability weighted moments} 
 by a linear combination. The probability weighted moments $ M_{p, r, s} $ are defined by
 \begin{equation}
 M_{p, r, s}= \int_{-\infty}^{\infty}x^{p}\left[ F\left( x\right) \right] ^{r}\left\lbrace 1-F\left( x\right)\right\rbrace ^{s} f\left( x\right)dx.
 \end{equation}
Observe that $ M_{p, 0, 0} $ represents the conventional noncentral  moments. We shall use the quantities $ M_{1, r, 0} $  where the random variable $ x $ enters linearly. In particular, we define $ \beta_{r} = M_{1, r, 0}$ as the probability weighted moments. The $ \beta_{r}\mathrm{'s} $ find application, for example, in evaluating  the moments of order statisitcs (discuused in Anis and Dey (2020)). The linear combination between the $ L-\mathrm{moments} $ (denoted by $ \lambda_{i} $) and the PWMs  $ \beta_{r} $ are given below for the first four moments:
\begin{eqnarray}
\lambda_{1} & = & \beta_{0}\\
\lambda_{2} & = & 2\beta_{1}-\beta_{0}\\
\lambda_{3} & = & 6\beta_{2}-6\beta_{1}+\beta_{0}\\
\lambda_{4} & = & 20\beta_{3} - 30\beta_{2}+12\beta_{1}-\beta_{0}.
\end{eqnarray}
In the particular case of the unit-Gompertz distribution, after routine calculation, we  find that  the $ r\mathrm{-th} $ PWM is given by $$ \beta_{r}=\alpha^{1/\beta}\left( r+1\right) ^{\frac{1}{\beta}-1} e^{\left( r+1\right) \alpha}\Gamma \left( 1-\frac{1}{\beta}; \left( r+1\right) \alpha \right) ;$$
and hence the  the $ L-\mathrm{moments} $ can be obtained. It should be noted that the algebraic expressions are rather involved; but for given values of the parameters $ \alpha $ and $ \beta, $ these  $ L-\mathrm{moments} $ can be easily obtained numerically.

\subsection{Entropy}

Entropy is used to measure the amount of information (or uncertainty)
contained in a random observation regarding its parent distribution (population). A large value of entropy implies  greater uncertainty in the data. Since its introduction by Shannon (1948), it has witnessed many generalizations. Anis and De (2020) have discussed the Shannon and R{\'e}nyi entropies for the unit-Gompertz distribution. Here we look at two other genralizations, namely the Tsallis and Mathai–Haubold entropies.
\subsubsection{The Tsallis Entropy}
The Tsallis entropy was introduced by Tsallis (1988) and is defined by 
\begin{equation}
I_{T}\left( \gamma\right) =\frac{1}{\gamma-1}\left( 1-\int_{-\infty}^{\infty}\left[ f\left( x\right) \right] ^{\gamma}dx\right), 0<\gamma \neq 1. 
\end{equation}
Clearly, the Tsallis entropy reduces to the classical Shannon entropy as $ \gamma \rightarrow 1.$ There are many applications of the Tsallis entropy. In physics,  it is used to describe a
number of non-extensive systems (Hamity and Barraco, 1996). It has found application in image processing (Yu et al., 2009)
and signal processing (Tong et al., 2002). Zhang et al. (2010) use a Tsallis entropy -based measure to reveal the presence and the extent of development of burst suppression activity following brain injury. Zhang and Wu (2011)  use the Tsallis entropy to propose a global multi-level thresholding method for image 
segmentation.  \\
For the unit-Gompertz distribution, the Tsallis entropy is given by  
\begin{equation}
I_{T}\left( \gamma\right) =\frac{1}{\gamma-1}\left[ 1 - \frac{\alpha^{\frac{1-\gamma}{\beta}}\beta^{\gamma-1}}{\gamma^{\frac{\gamma+\beta\gamma-1}{\beta}}}e^{\alpha\gamma}\Gamma \left( \frac{\gamma +\beta\gamma-1}{\beta}; \alpha\gamma\right)  \right] , 0<\gamma \neq 1, 
\end{equation}
where $ \Gamma \left( s; x\right)  $ is defined by (\ref{inc-up-gamma}) above.

\subsubsection{The Mathai–Haubold Entropy}

Mathai and Haubold (2008) introduced a new measure of entropy. It is defined by 
\begin{equation}
I_{MH}\left( \gamma\right) = \frac{1}{\gamma-1 }\left( \int_{-\infty}^{\infty}\left[ f\left( x\right) \right] ^{2- \gamma}dx -1\right), \gamma \neq 1, \gamma <2. 
\end{equation}

The entropy  $ I_{MH}\left( \gamma\right) $ is an inaccuracy measure through disturbance or distortion of systems. In case of the unit-Gompertz distribution, the Mathai–Haubold entropy is given by 
\begin{eqnarray}
I_{T}\left( \gamma\right) & =& \frac{1}{\gamma-1}\left[\frac{\alpha^{\frac{\left( 1-\beta\right) \left( \gamma-1+\beta\right) }{\beta}}\beta^{1-\gamma}}{\left( 2-\gamma\right) ^{\frac{1+\beta-\beta^{2}}{\beta}}}e^{\alpha\left( 2-\gamma\right) }\Gamma \left( \frac{\left( 1+\beta\right) \left( 1-\gamma\right)}{\beta}+1; \alpha\left( 2-\gamma\right) \right) - 1  \right] , \nonumber \\
& & {} \mbox{~~~~~~~~~~~~~~~~~~~~~~~~~~~~~~~~~~~~~~~~~~~~~~~~~~~~~~}\gamma \neq 1, \gamma <2,
\end{eqnarray}
where $ \Gamma \left( s; x\right)  $ is defined by (\ref{inc-up-gamma}) above.

	\section{Conclusion}

	In this work we have presented two characterizations of the recently-introduced unit-Gompertz distribution based on truncated moments. To the best of our knowledge this is the only characterization of this distribution available in the literature till date. We hope this will enable researchers to understand whether the given data at hand can be modelled by this distribution. We also looked at the $ L- $moments and two measures of entropy.


\end{document}